%% file: fsf.tex
\documentclass[english,10pt]{amsart}

\usepackage{amsmath, amssymb, amsthm}
\usepackage[all]{xy}
\usepackage[activeacute]{babel}
\newtheorem{thm}[equation]{Theorem}
\newtheorem{lem}[equation]{Lemma}
\newtheorem{prop}[equation]{Proposition}
\newtheorem{cor}[equation]{Corollary}

\theoremstyle{remark}
\newtheorem{rmk}[equation]{Remark}

\theoremstyle{definition}
\newtheorem{defi}[equation]{Definition}

\newcommand{\TspectraX}{Spt(\mathcal M _{X})}

\newcommand{\stablehomotopyX}{\mathcal{SH}_{X}}
\newcommand{\stablehomotopyY}{\mathcal{SH}_{Y}}

\newcommand{\stablehomotopyXprime}{\mathcal{SH}_{X'}}
\newcommand{\stablehomotopyW}{\mathcal{SH}_{W}}
\newcommand{\stablehomotopyk}{\mathcal{SH}_{k}}
\newcommand{\Hom}{\mathrm{Hom}}
\newcommand{\qorthogonalX}{\stablehomotopyX ^{\perp}(q)}
\newcommand{\qorthogonalY}{\stablehomotopyY ^{\perp}(q)}
\newcommand{\qplusoneorthogonalX}{\stablehomotopyX ^{\perp}(q+1)}
\newcommand{\qplusoneorthogonalY}{\stablehomotopyY ^{\perp}(q+1)}
\newcommand{\qplusoneorthogonalW}{\stablehomotopyW ^{\perp}(q+1)}
\newcommand{\qplusoneorthogonalinvY}{\mathcal{SH}^{\perp}_{p^{-1}Y}(q+1)}
\newcommand{\slicefilt}{\rm{sf}}

\numberwithin{equation}{section}

\begin{document}

%%%%%%%%%%%%%%%%%%%%%%%%%%%%%%%%%%%%%%%%%%%%%%%
%%%%%%%%%%%%%%%%%%%%%%%%%%%%%%%%%%%%%%%%%%%%%%%
%%%%%%%%%%%%%%%%%%%%%%%%%%%%%%%%%%%%%%%%%%%%%%%

\title{On the Functoriality of the Slice Filtration}

%%%%%%%%%%%%%%%%%%%%%%%%%%%%%%%%%%%%%%%%%%%%%%%
%%%%%%%%%%%%%%%%%%%%%%%%%%%%%%%%%%%%%%%%%%%%%%%
%%%%%%%%%%%%%%%%%%%%%%%%%%%%%%%%%%%%%%%%%%%%%%%

\author{Pablo Pelaez}
\address{Universit\"at Duisburg-Essen, Mathematik, 45117 Essen, Germany}
\email{pablo.pelaez@uni-due.de}

%%%%%%%%%%%%%%%%%%%%%%%%%%%%%%%%%%%%%%%%%%%%%%%
%%%%%%%%%%%%%%%%%%%%%%%%%%%%%%%%%%%%%%%%%%%%%%%
%%%%%%%%%%%%%%%%%%%%%%%%%%%%%%%%%%%%%%%%%%%%%%%



%%%%%%%%%%%%%%%%%%%%%%%%%%%%%%%%%%%%%%%%%%%%%%%
%%%%%%%%%%%%%%%%%%%%%%%%%%%%%%%%%%%%%%%%%%%%%%%
%%%%%%%%%%%%%%%%%%%%%%%%%%%%%%%%%%%%%%%%%%%%%%%

\begin{abstract}
	Let $k$ be a field with resolution of singularities, and $X$ a
	separated $k$-scheme of finite type
	with structure map $g$.  We show that the slice filtration
	in the motivic stable homotopy category
	commutes with pullback along $g$.  Restricting the field further
	to the case of characteristic zero, we are able to
	compute the slices of
	Weibel's
	homotopy invariant $K$-theory \cite{MR991991}
	extending the result of
	Levine \cite{MR2365658}, and also the zero slice of the sphere
	spectrum extending the result of Levine \cite{MR2365658}
	and Voevodsky \cite{MR2101286}.  We also show that the
	zero slice of the sphere spectrum is a strict cofibrant ring spectrum
	$\mathbf{HZ}_{X}^{\slicefilt}$
	which is stable under pullback and that all the slices have
	a canonical structure of strict modules over $\mathbf{HZ}_{X}^{\slicefilt}$.
	If we consider rational coefficients and assume that $X$
	is geometrically unibranch then relying on the work
	of Cisinski and D{\'e}glise \cite{mixedmotives}, we deduce that
	the zero slice of the sphere spectrum is given by Voevodsky's
	rational motivic cohomology spectrum $\mathbf{HZ}_{X}\otimes \mathbb Q$
	and that the slices have transfers.  This proves several conjectures
	of Voevodsky \cite[conjectures 1, 7, 10, 11]{MR1977582}
	in characteristic zero. 
\end{abstract}

%%%%%%%%%%%%%%%%%%%%%%%%%%%%%%%%%%%%%%%%%%%%%%%
%%%%%%%%%%%%%%%%%%%%%%%%%%%%%%%%%%%%%%%%%%%%%%%
%%%%%%%%%%%%%%%%%%%%%%%%%%%%%%%%%%%%%%%%%%%%%%%

\maketitle

%%%%%%%%%%%%%%%%%%%%%%%%%%%%%%%%%%%%%%%%%%%%%%%
%%%%%%%%%%%%%%%%%%%%%%%%%%%%%%%%%%%%%%%%%%%%%%%
%%%%%%%%%%%%%%%%%%%%%%%%%%%%%%%%%%%%%%%%%%%%%%%

\input{introd}
\input{sect1_fsf}
\input{sect2_fsf}
\input{sect3_fsf}

%%%%%%%%%%%%%%%%%%%%%%%%%%%%%%%%%%%%%%%%%%%%%%%
%%%%%%%%%%%%%%%%%%%%%%%%%%%%%%%%%%%%%%%%%%%%%%%
%%%%%%%%%%%%%%%%%%%%%%%%%%%%%%%%%%%%%%%%%%%%%%%
\section*{Acknowledgements}
	The author would like to thank Denis-Charles Cisinski for several conversations
	and suggestions which led to some of the arguments in sections 
	\ref{sect-2} and \ref{sect-3}, Marc Levine
	for pointing out several errors in a previous version of this note and
	Fr{\'e}d{\'e}ric D{\'e}glise for suggesting this problem.

%%%%%%%%%%%%%%%%%%%%%%%%%%%%%%%%%%%%%%%%%%%%%%%
%%%%%%%%%%%%%%%%%%%%%%%%%%%%%%%%%%%%%%%%%%%%%%%
%%%%%%%%%%%%%%%%%%%%%%%%%%%%%%%%%%%%%%%%%%%%%%%
\bibliography{biblio_fsf}
\bibliographystyle{abbrv}

\end{document}

%% file: introd.tex
%%%%%%%%%%%%%%%%%%%%%%%%%%%%%%%%%%%%%%%%%%%%%%%%%%%%%%%
%%%%%%%%%%%%%%%%%%%%%%%%%%%%%%%%%%%%%%%%%%%%%%%%%%%%%%%
%%%%%%%%%%%%%%%%%%%%%%%%%%%%%%%%%%%%%%%%%%%%%%%%%%%%%%%
\begin{section}{Introduction}
		\label{Introd}
		
	The goal of this paper is
	the study of the behavior with respect to pullback of 
	the slice filtration introduced by Voevodsky in motivic 
	homotopy theory \cite{MR1977582}.
	We introduce a general criterion (see Theorem \ref{thm.criterion})
	which guarantees that the slice filtration commutes
	with pullback and verify that it holds (see Theorem \ref{thm.compatibility-with-res-sing})
	on the category of schemes of finite
	type (not necessarily smooth) over a field $k$ with resolution of singularities.
	
	In the last section of the paper some interesting applications are given for
	base schemes over a field $k$ of characteristic zero.  Among them,
	we are able to compute the zero slice of the sphere spectrum 
	(see Theorem \ref{thm.computation-slices}\eqref{thm.computation-slices.1})
	extending a result of
	Levine \cite{MR2365658} and Voevodsky \cite{MR2101286}, 
	and all the slices of Weibel's homotopy invariant K-theory
	(see Theorem \ref{thm.computation-slices}\eqref{thm.computation-slices.4})
	extending a result of Levine \cite{MR2365658}.
	This allows us to introduce a family of triangulated categories given by the homotopy
	category associated to the category of strict modules over the zero slice of the sphere
	spectrum (see Definition \ref{def.slicefilt-motives}), 
	which provide a natural framework for a theory of mixed motives over
	the category of $k$-schemes of finite type, since the construction:
		\begin{enumerate}
			\item is naturally equipped with the formalism of Grothendieck's six operations
						(see Theorem \ref{thm.motivic-cats}).
			\item is naturally equivalent to Voevodsky's triangulated category of motives 
				when the base scheme is a field
				(see Theorem \ref{thm.comparing-motcats}),
				this holds with integral coefficients so the construction may be a useful
				tool for the study of torsion
				in motivic cohomology.
			\item is  equipped with a canonical spectral sequence converging to Weibel's homotopy
				invariant $K$-theory.  This follows from
				our computation of the slices for homotopy invariant $K$-theory.
		\end{enumerate}	
		
\subsection*{Notation}
		
	In all the categories under consideration, $0$ will be the final object and
	$\cong$ will denote that two objects are isomorphic.	
	
	Let $X$ be a Noetherian separated scheme of finite Krull dimension, and $\mathcal M _{X}$
	be the category of pointed simplicial presheaves in the smooth Nisnevich site $Sm_{X}$ over $X$ equipped
	with the motivic Quillen model structure
	\cite{MR0223432} 
	introduced in \cite[Thm. A.17]{MR2597741}.  
	We define $T_{X}$ in $\mathcal M_{X}$ as the pointed simplicial presheaf
	represented by $S^{1}\wedge \mathbb G_{m}$, where
	$\mathbb G_{m}$ is the multiplicative group $\mathbb A^{1}_{X}-\{ 0 \}$ pointed by $1$, and $S^{1}$
	denotes the simplicial circle.  Given an arbitrary integer $r\geq 1$,
	$S^{r}$ (respectively $\mathbb G _{m}^{r}$) will denote the iterated smash
	product $S^{1}\wedge \cdots \wedge S^{1}$ (respectively
	$\mathbb G _{m}\wedge \cdots \wedge \mathbb G _{m}$) with $r$-factors;
	by definition, $S^{0}=\mathbb G _{m}^{0}$ will be the pointed simplicial
	presheaf $X_{+}$ represented by the base scheme $X$.
	We will write $T^{r}_{X}$ for $S^{r}\wedge \mathbb G _{m}^{r}$.
	
	Let $Spt(\mathcal M _{X})$ denote Jardine's category of symmetric $T_{X}$-spectra on 
	$\mathcal M _{X}$ equipped with the motivic model structure defined in 
	\cite[Thm. A.38]{MR2597741} and
	$\stablehomotopyX$ denote its homotopy category, which is triangulated.
	
	For every integer $q\in \mathbb Z$, we consider the following family of symmetric $T_{X}$-spectra
	
		\[	C^{q}_{\mathit{eff}}(X)=\{ F_{n}(S^{r}\wedge \mathbb G _{m}^{s}\wedge U_{+}) 
			\mid n,r,s \geq 0; s-n\geq q; U\in Sm_{X}\}
		\]
	where $F_{n}$ is the left adjoint to the $n$-evaluation functor
		\[	\xymatrix@R=0.5pc{Spt(\mathcal M _{X}) \ar[r]^-{ev_{n}}& \mathcal M _{X}\\
							(E^{m})_{m\geq 0}\ar@{|->}[r] & E^{n}}
		\]		

	Voevodsky \cite{MR1977582} defines the slice filtration as the following family of triangulated subcategories of $\stablehomotopyX$
		\[	\cdots \subseteq \Sigma _{T}^{q+1}\stablehomotopyX^{\mathit{eff}} \subseteq \Sigma _{T}^{q}
			\stablehomotopyX^{\mathit{eff}}
			\subseteq \Sigma _{T}^{q-1}\stablehomotopyX^{\mathit{eff}} \subseteq \cdots
		\]
	where $\Sigma _{T}^{q}\stablehomotopyX^{\mathit{eff}}$ is the smallest full triangulated subcategory of 
	$\stablehomotopyX$ which contains
	$C^{q}_{\mathit{eff}}(X)$ and is closed under arbitrary coproducts.

	It follows from the work of Neeman \cite{MR1308405}, \cite{MR1812507} that the inclusion
		\[	i_{q}:\Sigma _{T}^{q}\stablehomotopyX^{\mathit{eff}}\rightarrow \stablehomotopyX
		\]
	has a right adjoint $r_{q}:\stablehomotopyX \rightarrow \Sigma _{T}^{q}\stablehomotopyX^{\mathit{eff}}$, 
	and that the following functors
		\begin{align*}
				f_{q} & : \stablehomotopyX \rightarrow \stablehomotopyX \\
				s_{q} & :\stablehomotopyX \rightarrow \stablehomotopyX
		\end{align*}
	are triangulated, where $f_{q}$ is defined as the 
	composition $i_{q}\circ r_{q}$, and $s_{q}$ is characterized by the fact
	that for every $E\in \stablehomotopyX$, we have the following distinguished triangle in $\stablehomotopyX$
		\[	\xymatrix{f_{q+1}E \ar[r]^-{\rho _{q}^{E}}& f_{q}E \ar[r]^-{\pi _{q}^{E}}& s_{q}E 
			\ar[r]& S^{1}\wedge f_{q+1}E}
		\]
	We will refer to $f_{q}E$ as the $(q-1)$-connective cover of $E$, and to $s_{q}E$ as the $q$-slice of $E$.
	It follows directly from the definition that the $q$-slice of $E$ satisfies the following property:
		\[	\Hom _{\stablehomotopyX}(K,s_{q}E)=0
		\]
	for every symmetric $T_{X}$-spectrum $K$ in $\Sigma _{T}^{q+1}\stablehomotopyX^{\mathit{eff}}$.

\end{section}

%% file: sect1_fsf.tex
%%%%%%%%%%%%%%%%%%%%%%%%%%%%%%%%%%%%%%%%%%%%%%%%%%%%%%%%%
%%%%%%%%%%%%%%%%%%%%%%%%%%%%%%%%%%%%%%%%%%%%%%%%%%%%%%%%%
%%%%%%%%%%%%%%%%%%%%%%%%%%%%%%%%%%%%%%%%%%%%%%%%%%%%%%%%%
\begin{section}{A general criterion}
		\label{sect-1}

	In the rest of this section $g:X\rightarrow Y$ will be a map of schemes, 
	where $X$ and $Y$ are Noetherian, separated
	and of finite Krull dimension.  Our goal is to
	introduce a general criterion which implies the compatibility between
	the slice filtration and pullback along $g$.
		
	The $2$-functor 
		\[	X\mapsto \stablehomotopyX
		\]
	is homotopic stable in the sense of Ayoub \cite[chapter 4]{MR2438151}
	and in particular is equipped with the formalism of Grothendieck's
	six operations \cite[Scholium 1.4.2]{MR2423375}.  Hence, given a map
	$g:X\rightarrow Y$ of schemes, there exists a pair of adjunctions between
	triangulated functors:
		\begin{align*}	(\mathbf L g^{\ast}, \mathbf R g_{\ast}, \varphi):
				\stablehomotopyY \rightarrow \stablehomotopyX \\
				(g_{!}, g^{!}, \psi):
				\stablehomotopyX \rightarrow \stablehomotopyY
		\end{align*}
	%The adjunction $(\mathbf L g^{\ast}, \mathbf R g_{\ast}, \varphi)$ is the Kan extension
	%of the adjunction \cite[Prop. 4.5.4]{MR2438151}
	%	\[	(g^{\ast}, g_{\ast}, \varphi): \TspectraY \rightarrow \TspectraX
	%	\]
	%and 
	where the functor $\mathbf L g^{\ast}$ is characterized by the following property:
	Given $U\in Sm_{Y}$,
	$\mathbf L g^{\ast}(F_{0}(U_{+}))=g^{\ast}(F_{0}(U_{+}))=F_{0}(X\times _{Y}U _{+})$.
	
	If $g:X\rightarrow Y$ is a smooth map of finite type, the functor
	$\mathbf L g^{\ast}$ admits a left adjoint 
			\[	\mathbf L g_{\sharp} :\stablehomotopyX \rightarrow 	
				\stablehomotopyY
			\]
	which is also triangulated, and is characterized by the following property:
	Given $U\in Sm_{X}$ with structure map $u$,
	$\mathbf L g_{\sharp}(F_{0}(U_{+}))=F_{0}(U_{+})$,
	where we consider $U$ as a scheme over $Y$ with structure map
	$g\circ u$ (see \cite[Prop. 1.23(2)]{MR1813224}).
	
	Furthermore, these functors satisfy the localization axiom:
	
	\begin{thm}
			\label{thm.localization}
		Let $i:Z\rightarrow X$ be a closed immersion, and $j:U\rightarrow X$ its
		open complement.  Then for every symmetric 
		$T_{X}$-spectrum $E\in \stablehomotopyX$, there exists
		a canonical distinguished triangle in $\stablehomotopyX$:
			\[	\mathbf L j_{\sharp}\; \mathbf L j^{\ast}E\rightarrow
				E\rightarrow  \mathbf R i_{\ast} \; \mathbf L i^{\ast}E
				\rightarrow S^{1}\wedge \mathbf L j_{\sharp}\; \mathbf L j^{\ast}E
			\]
	\end{thm}
	\begin{proof}
		We refer the reader to \cite[\S 4.5.3]{MR2438151}.
	\end{proof}
	
	Consider the following fibred product diagram:
		\[	\xymatrix{X' \ar[d]_{l} \ar[r]^-{k}& X \ar[d]^-{g}\\
							Y' \ar[r]_-{h}& Y}
		\]
		
	\begin{prop}
			\label{prop.sharp.pfwd.comp}
		If $g$ is a proper map, and $h$ is an open immersion, then for every
		$E\in \stablehomotopyXprime$ there exists a canonical isomorphism
		 	\[	\mathbf R g_{\ast}\; \mathbf L k_{\sharp} \; E\rightarrow
				\mathbf L h_{\sharp} \; \mathbf R l_{\ast} \; E
			\]
		 in $\stablehomotopyY$.
	\end{prop}
	\begin{proof}
		We observe that $h$ and $k$ are open immersions.  Hence, by
		\cite[Scholium 1.4.2(3)]{MR2423375} there exist natural isomorphisms:
			\begin{align*}
				\mathbf L h_{\sharp}& \rightarrow  h_{!}\\
				\mathbf L k_{\sharp}& \rightarrow  k_{!}
			\end{align*}
		On the other hand $g$ and $l$ are proper maps.  Therefore, by
		\cite[Scholium 1.4.2(4)]{MR2423375} and \cite[Thm. 2.2.14(1)]{mixedmotives}
		there exist natural isomorphisms:
			\begin{align*}
				\mathbf R g_{\ast}& \rightarrow  g_{!}\\
				\mathbf R l_{\ast}& \rightarrow  l_{!}
			\end{align*}
		Thus, we deduce that there exist the following isomorphisms in $\stablehomotopyY$:
			\begin{align*}	\mathbf R g_{\ast}\; \mathbf L k_{\sharp} \; E & \cong g_{!}\; k_{!}\; E\\
									\mathbf L h_{\sharp} \; \mathbf R l_{\ast} \; E & \cong h_{!}\; l_{!}\; E
			\end{align*}
		Finally, by functoriality we conclude that
		$g_{!} k_{!} E$ and $h_{!} l_{!} E$ are isomorphic in $\stablehomotopyY$.
		This finishes the proof.
	\end{proof}
			
%%%%%%%%%%%%%%%%%%%%%%%%%%%%%%%%%%%%%%%%%%%%%%%%%%%%%%%%%%%%%%%%%%
%%%%%%%%%%%%%%%%%%%%%%%%%%%%%%%%%%%%%%%%%%%%%%%%%%%%%%%%%%%%%%%%%%
%%%%%%%%%%%%%%%%%%%%%%%%%%%%%%%%%%%%%%%%%%%%%%%%%%%%%%%%%%%%%%%%%%
	\begin{lem}
			\label{lem.pullbackpreservesconnectivity}
		Let $q\in \mathbb Z$ be an arbitrary integer.  Then
			\[	\mathbf L g^{\ast}(\Sigma _{T}^{q}\stablehomotopyY^{\mathit{eff}})\subseteq \Sigma _{T}^{q}
		\stablehomotopyX^{\mathit{eff}}
		\]
		i.e.
		the functor $\mathbf L g^{\ast}:\stablehomotopyY \rightarrow \stablehomotopyX$ respects connective objects.
	\end{lem}	
	\begin{proof}
		This follows directly from the fact that $g^{\ast}(T_{Y})=T_{X}$.
	\end{proof}

	It follows immediately from Lemma \ref{lem.pullbackpreservesconnectivity} that for
	any integer $q\in \mathbb Z$, there exists a pair of natural transformations
		\begin{align*}
			\alpha _{q} & : \mathbf L g^{\ast}\circ f_{q}\rightarrow f_{q}\circ \mathbf L g^{\ast}\\
			\beta _{q} & : \mathbf L g^{\ast}\circ s_{q}\rightarrow s_{q}\circ \mathbf L g^{\ast}
		\end{align*}
	such that for every $E\in \stablehomotopyY$ the following diagram
		\begin{equation}
				\label{equation.naturaltransformations}
			\begin{array}{c}
				\xymatrix{\mathbf L g^{\ast}(f_{q+1}E) \ar[d]^-{\alpha _{q+1}(E)} \ar[r]^-{\mathbf L g^{\ast}(\rho _{q}^{E})}& 
									\mathbf L g^{\ast}(f_{q}E) \ar[d]^-{\alpha _{q}(E)} \ar[r]^-{\mathbf L g^{\ast}(\pi _{q}^{E})}& 
									\mathbf L g^{\ast}(s_{q}E) \ar[d]_-{\beta _{q}(E)} \ar[r] & 
									S^{1}\wedge \mathbf L g^{\ast}(f_{q+1}E) 
									\ar[d]_-{id\wedge \alpha _{q+1}(E)}\\
									f_{q+1}(\mathbf L g^{\ast}E) \ar[r]_-{\rho _{q}^{\mathbf L g^{\ast}E}}& 
									f_{q}(\mathbf L g^{\ast}E) \ar[r]_-{\pi _{q}^{\mathbf L g^{\ast}E}}& s_{q}
									(\mathbf L g^{\ast}E) \ar[r]& S^{1}\wedge f_{q+1}(\mathbf L g^{\ast}E)}
			\end{array}
		\end{equation}
	is commutative and its rows are distinguished triangles in $\stablehomotopyX$.
	
%%%%%%%%%%%%%%%%%%%%%%%%%%%%%%%%%%%%%%%%%%%%%%%%%%%%%%%%%%%%%%%%%%
%%%%%%%%%%%%%%%%%%%%%%%%%%%%%%%%%%%%%%%%%%%%%%%%%%%%%%%%%%%%%%%%%%
%%%%%%%%%%%%%%%%%%%%%%%%%%%%%%%%%%%%%%%%%%%%%%%%%%%%%%%%%%%%%%%%%%
	\begin{defi}
			\label{def.compatibility-with-pullbacks}
		We say that the slice filtration is \emph{compatible with pullbacks} along $g$,
		if $\beta _{q}$ is a natural isomorphism for every $q\in \mathbb Z$.
	\end{defi}
	
	\begin{lem}
			\label{lemma.pre.criterion}
		Let $E\in \stablehomotopyY$ be a symmetric $T_{Y}$-spectrum
		and $q\in \mathbb Z$.  Then the natural map:
		\[	\xymatrix{\alpha _{q}(f_{q}E): \mathbf L g^{\ast}(f_{q}f_{q}E) \ar[r] & 
							f_{q}(\mathbf L g^{\ast}(f_{q}E))}
		\]
		is an isomorphism in $\stablehomotopyX$.
	\end{lem}
	\begin{proof}
		By construction $\alpha _{q}(f_{q}E)$ fits in the following commutative diagram:
				\[	\xymatrix{\mathbf L g^{\ast}(f_{q}f_{q}E) \ar[d]_-{\alpha _{q}(f_{q}E)} 
									\ar[drr]^-{\mathbf L g^{\ast}(\theta ^{f_{q}E})} &&\\
									f_{q}(\mathbf L g^{\ast}f_{q}E) \ar[rr]_-{\theta ^{\mathbf L g^{\ast}f_{q}E}}
									&& \mathbf L g^{\ast}f_{q}E}
				\]
		where $\theta$ denotes the counit of the adjunction
			\[	(i_{q}, r_{q}):\Sigma _{T}^{q}\stablehomotopyX^{\mathit{eff}}\rightarrow \stablehomotopyX
			\]
		Thus, it suffices to show that $\mathbf L g^{\ast}(\theta ^{f_{q}E})$,
		$\theta ^{\mathbf L g^{\ast}f_{q}E}$ are isomorphisms in $\stablehomotopyX$.
			
		We observe that by construction $\theta ^{f_{q}E}$ is an isomorphism in $\stablehomotopyY$,
		hence $\mathbf L g^{\ast}(\theta ^{f_{q}E})$ is an isomorphism in $\stablehomotopyX$.
		Finally,
		it follows from Lemma \ref{lem.pullbackpreservesconnectivity} 
		that $\theta ^{\mathbf L g^{\ast}f_{q}E}$ is an isomorphism in $\stablehomotopyX$.
	\end{proof}	
		
%%%%%%%%%%%%%%%%%%%%%%%%%%%%%%%%%%%%%%%%%%%%%%%%%%%%%%%%%%%%%%%%%%
%%%%%%%%%%%%%%%%%%%%%%%%%%%%%%%%%%%%%%%%%%%%%%%%%%%%%%%%%%%%%%%%%%
%%%%%%%%%%%%%%%%%%%%%%%%%%%%%%%%%%%%%%%%%%%%%%%%%%%%%%%%%%%%%%%%%%
	\begin{defi}
			\label{def.orthoganality}
		Let $E \in \stablehomotopyX$ be a symmetric $T_{X}$-spectrum and $q\in \mathbb Z$.
		We say that $E$ is \emph{$q$-orthogonal with respect to the slice filtration 
		in $\stablehomotopyX$}, if one of the following
		equivalent conditions holds:
			\begin{enumerate}
				\item	\label{orthogonality.1}  $f_{q}E=0$.
				\item	\label{orthogonality.2} $\Hom _{\stablehomotopyX}(F,E)=0$
								for every $F\in \Sigma _{T}^{q}\stablehomotopyX^{\mathit{eff}}$.
			\end{enumerate}			
		Let $\qorthogonalX$ denote the full subcategory of $\stablehomotopyX$
		consisting of the  
		symmetric $T_{X}$-spectra which are $q$-orthogonal
		with respect to the slice filtration in $\stablehomotopyX$.
	\end{defi}

%%%%%%%%%%%%%%%%%%%%%%%%%%%%%%%%%%%%%%%%%%%%%%%%%%%%%%%%%
%%%%%%%%%%%%%%%%%%%%%%%%%%%%%%%%%%%%%%%%%%%%%%%%%%%%%%%%%
%%%%%%%%%%%%%%%%%%%%%%%%%%%%%%%%%%%%%%%%%%%%%%%%%%%%%%%%%
	\begin{lem}	
			\label{lemma.orthognal-triangulated}
		$\qorthogonalX$ is a triangulated subcategory of $\stablehomotopyX$.
	\end{lem}
	\begin{proof}
		It follows immediately from the fact that the functor $\Hom _{\stablehomotopyX}(A,-)$
		is homological (see \cite[Def. 1.1.7]{MR1812507})
		for every $A\in \stablehomotopyX$.
	\end{proof}
			
%%%%%%%%%%%%%%%%%%%%%%%%%%%%%%%%%%%%%%%%%%%%%%%%%%%%%%%%%
%%%%%%%%%%%%%%%%%%%%%%%%%%%%%%%%%%%%%%%%%%%%%%%%%%%%%%%%%
%%%%%%%%%%%%%%%%%%%%%%%%%%%%%%%%%%%%%%%%%%%%%%%%%%%%%%%%%
	\begin{lem}	
			\label{lemma.pushforward-respects-orthogonal}
		The functor
		$\mathbf R g_{\ast}$ is compatible with the $q$-orthogonal objects with respect to the
		slice filtration, i.e.
			\[	\mathbf R g_{\ast}(\qorthogonalX)\subseteq \qorthogonalY
			\]
	\end{lem}
	\begin{proof}
		This follows directly from adjointness and Lemma \ref{lem.pullbackpreservesconnectivity}.
	\end{proof}
	
%%%%%%%%%%%%%%%%%%%%%%%%%%%%%%%%%%%%%%%%%%%%%%%%%%%%%%%%%
%%%%%%%%%%%%%%%%%%%%%%%%%%%%%%%%%%%%%%%%%%%%%%%%%%%%%%%%%
%%%%%%%%%%%%%%%%%%%%%%%%%%%%%%%%%%%%%%%%%%%%%%%%%%%%%%%%%
	\begin{lem}
			\label{lemma.criterion}
		Let $E\in \stablehomotopyY$ be a symmetric $T_{Y}$-spectrum
		and $q\in \mathbb Z$.
		If the following condition holds:
			\begin{equation}
					\label{equation.orthogonalitycondition}
				\mathbf L g^{\ast}(s_{q}E)\in \qplusoneorthogonalX
			\end{equation}	
		then the natural maps:
			\[	\xymatrix@R=0.5pt{\alpha _{q+1}(f_{q}E):\mathbf L g^{\ast}(f_{q+1}f_{q}E) \ar[r] & f_{q+1}
													(\mathbf L g^{\ast}(f_{q}E))  \\
				\beta _{q}(f_{q}E): \mathbf L g^{\ast}(s_{q}f_{q}E) \ar[r] & s_{q}(\mathbf L g^{\ast}(f_{q}E))}
			\]
		are isomorphisms in $\stablehomotopyX$.
	\end{lem}
	\begin{proof}
		Consider the commutative diagram (\ref{equation.naturaltransformations}) for $f_{q}E$:
				\[	\xymatrix@C=-0.1pc{\mathbf L g^{\ast}(f_{q+1}f_{q}E) \ar[d]_-{\alpha _{q+1}(f_{q}E)} 
					\ar[rrrrrrrr]^-{\mathbf L g^{\ast}(\rho _{q}^{f_{q}E})} &&&&&&&&
									\mathbf L g^{\ast}(f_{q}f_{q}E) \ar[d]_-{\alpha _{q}(f_{q}E)} 
									\ar[rrrrrrrr]^-{\mathbf L g^{\ast}(\pi _{q}^{f_{q}E})} &&&&&&&&
									\mathbf L g^{\ast}(s_{q}f_{q}E) \ar[d]^-{\beta _{q}(f_{q}E)} \ar[rrr] &&&
									S^{1}\wedge \mathbf L g^{\ast}(f_{q+1}f_{q}E) \ar[d] \\
									f_{q+1}(\mathbf L g^{\ast}f_{q}E) \ar[rrrrrrrr]_-{\rho _{q}^{\mathbf L 
									g^{\ast}f_{q}E}} &&&&&&&&
									f_{q}(\mathbf L g^{\ast}f_{q}E) 
									\ar[rrrrrrrr]_-{\pi _{q}^{\mathbf L g^{\ast}f_{q}E}} &&&&&&&&
									s_{q}(\mathbf L g^{\ast}f_{q}E) 
									\ar[rrr] &&&
									S^{1}\wedge f_{q+1}(\mathbf L g^{\ast}f_{q}E)}
				\]
		By Lemma \ref{lemma.pre.criterion},
		$\alpha _{q}(f_{q}E)$ is an isomorphism.  Using the octahedral axiom,
		we deduce that
		the following diagram commutes and all its rows and columns are
		distinguished triangles in $\stablehomotopyX$:
			\[	\xymatrix@C=-0.1pc{\mathbf L g^{\ast}(f_{q+1}f_{q}E) \ar[d]_-{\alpha _{q+1}(f_{q}E)} 
					\ar[rrrrrrrr]^-{\mathbf L g^{\ast}(\rho _{q}^{f_{q}E})} &&&&&&&& 
									\mathbf L g^{\ast}(f_{q}f_{q}E) \ar[d]_-{\alpha _{q}(f_{q}E)} 
									\ar[rrrrrrrr]^-{\mathbf L g^{\ast}(\pi _{q}^{f_{q}E})} &&&&&&&& 
									\mathbf L g^{\ast}(s_{q}f_{q}E) \ar[d]^-{\beta _{q}(f_{q}E)} \ar[rrr] &&&
									S^{1}\wedge \mathbf L g^{\ast}(f_{q+1}f_{q}E) \ar[d] \\
									f_{q+1}(\mathbf L g^{\ast}f_{q}E) \ar[rrrrrrrr]_-{\rho _{q}^{\mathbf L 
									g^{\ast}f_{q}E}} \ar[d] &&&&&&&& 
									f_{q}(\mathbf L g^{\ast}f_{q}E) \ar[rrrrrrrr]_-{\pi _{q}^{\mathbf L 
									g^{\ast}f_{q}E}} \ar[d]&&&&&&&& 
									s_{q}(\mathbf L g^{\ast}f_{q}E) \ar[rrr] \ar[d]&&& 
									S^{1}\wedge f_{q+1}(\mathbf L g^{\ast}f_{q}E) \ar[d]\\
									A \ar[rrrrrrrr] &&&&&&&& 
									0 \ar[rrrrrrrr] &&&&&&&& 
									S^{1}\wedge A \ar@{=}[rrr]&&&
									S^{1}\wedge A}
			\]
		Thus, it suffices to show that $S^{1}\wedge A \cong 0$ in $\stablehomotopyX$.
		It follows from Lemma \ref{lem.pullbackpreservesconnectivity} that $\mathbf L g^{\ast}(f_{q+1}f_{q}E)$
		is in $\Sigma _{T}^{q+1}\stablehomotopyX^{\mathit{eff}}$, and by construction
		$f_{q+1}(\mathbf L g^{\ast}f_{q}E)$ is also in $\Sigma _{T}^{q+1}\stablehomotopyX^{\mathit{eff}}$.
		Hence,
		$A$ and $S^{1}\wedge A$ are both
		in $\Sigma _{T}^{q+1}\stablehomotopyX^{\mathit{eff}}$.  
		
		On the other hand, by hypothesis 
		$\mathbf L g^{\ast}(s_{q}E)\cong \mathbf L g^{\ast}(s_{q}f_{q}E)$ is in $\qplusoneorthogonalX$;
		therefore, Lemma \ref{lemma.orthognal-triangulated} implies that 
		$S^{1}\wedge A$ is in $\qplusoneorthogonalX$, since
		$s_{q}(\mathbf L g^{\ast}f_{q}E)$ is in $\qplusoneorthogonalX$ by construction.
		
		Thus, we conclude that
			\[	\Hom _{\stablehomotopyX}(S^{1}\wedge A, S^{1}\wedge A)=0
			\]
		and from this it follows at once that
		$S^{1}\wedge A \cong 0$ in $\stablehomotopyX$, as we wanted.	
	\end{proof}
		
%%%%%%%%%%%%%%%%%%%%%%%%%%%%%%%%%%%%%%%%%%%%%%%%%%%%%%%%%%%%%%%%%%
%%%%%%%%%%%%%%%%%%%%%%%%%%%%%%%%%%%%%%%%%%%%%%%%%%%%%%%%%%%%%%%%%%
%%%%%%%%%%%%%%%%%%%%%%%%%%%%%%%%%%%%%%%%%%%%%%%%%%%%%%%%%%%%%%%%%%
	\begin{thm}
			\label{thm.criterion}
		If the condition (\ref{equation.orthogonalitycondition}) in Lemma \ref{lemma.criterion}
		holds for every 
		symmetric $T_{Y}$-spectrum in
		$\stablehomotopyY$ and for every integer $\ell \in \mathbb Z$, then
		the slice filtration is compatible with pullbacks along $g$, i.e.
		there exists a natural isomorphism 
			\[	\beta _{\ell}:\mathbf L g^{\ast}\circ s_{\ell}\rightarrow s_{\ell}\circ \mathbf L g^{\ast}
			\]
		for every $\ell \in \mathbb Z$.
	\end{thm}
	\begin{proof}
		Let $E$ be a symmetric $T_{Y}$-spectrum in $\stablehomotopyY$ and fix
		an integer $q\in \mathbb Z$.  Then 
		$E\cong \mathrm{hocolim} _{p\leq q}f_{p}E$, and since $\mathbf L g^{\ast}$
		and $s_{q}$ commute with filtered homotopy colimits we deduce that
		$\beta _{q}(E):\mathbf L g^{\ast}(s_{q}E)\rightarrow s_{q}(\mathbf L g^{\ast}E)$ is
		given by $\mathrm{hocolim} _{p\leq q}\beta _{q}(f_{p}E)$.  Hence, it suffices to show
		that $\beta _{q}(f_{p}E):\mathbf L g^{\ast}(s_{q}(f_{p}E)) \rightarrow s_{q}\mathbf L g^{\ast}(f_{p}E)$
		is an isomorphism in $\stablehomotopyX$ for every integer $p\leq q$.
			
		Lemma \ref{lemma.criterion} implies that $\beta _{q}(f_{q}E)$ is an isomorphism.
		We now proceed by induction, and assume that $\beta _{q}(f_{r}E)$ is
		an isomorphism for some $r\leq q$.  It only remains to show that in this situation,
		$\beta _{q}(f_{r-1}E)$ is also an isomorphism.  Consider the following commutative
		diagram in $\stablehomotopyX$:
			\[	\xymatrix{\mathbf L g^{\ast}(s_{q}(f_{r}E)) \ar[rr]^-{\beta _{q}(f_{r}E)} \ar[d]_-{\mathbf L 
						g^{\ast}s_{q}(\rho _{r-1}^{E})}&&
						s_{q}(\mathbf L g^{\ast}(f_{r}E)) \ar[d]^-{s_{q}\mathbf L g^{\ast}(\rho _{r-1}^{E})}\\
						\mathbf L g^{\ast}(s_{q}(f_{r-1}E)) \ar[rr]_-{\beta _{q}(f_{r-1}E)}&& s_{q}(\mathbf L 
						g^{\ast}(f_{r-1}E))}
			\]
		Since $r\leq q$, the left vertical map is an isomorphism and our induction hypothesis
		says that $\beta _{q}(f_{r}E)$ is also an isomorphism.  Thus, it is enough to check
		that $s_{q}\mathbf L g^{\ast}(\rho _{r-1}^{E})$ is an isomorphism in $\stablehomotopyX$.  Now,
		we observe that the following diagram in $\stablehomotopyX$ commutes:
			\[	\xymatrix{s_{q}(\mathbf L g^{\ast}(f_{r}E)) \ar[r]^-{\cong} \ar[d]_-{s_{q}\mathbf L g^{\ast}
						(\rho _{r-1}^{E})}&  
						s_{q}(\mathbf L g^{\ast}(f_{r}f_{r-1}E)) \ar[d]^-{s_{q}(\alpha _{r}(f_{r-1}E))}\\
						s_{q}(\mathbf L g^{\ast}(f_{r-1}E)) \ar[r]_-{\cong}& s_{q}(f_{r}(\mathbf L g^{\ast}
						(f_{r-1}E)))}
			\]
		where the rows are both canonical isomorphisms and the right vertical map
		is also an isomorphism by Lemma \ref{lemma.criterion}.  Thus, 
		we conclude that $s_{q}\mathbf L g^{\ast}(\rho _{r-1}^{E})$
		is an isomorphism in $\stablehomotopyX$.  This finishes the proof.		
	\end{proof}
	
%%%%%%%%%%%%%%%%%%%%%%%%%%%%%%%%%%%%%%%%%%%%%%%%%%%%%%%%%%%%%%%%%%
%%%%%%%%%%%%%%%%%%%%%%%%%%%%%%%%%%%%%%%%%%%%%%%%%%%%%%%%%%%%%%%%%%
%%%%%%%%%%%%%%%%%%%%%%%%%%%%%%%%%%%%%%%%%%%%%%%%%%%%%%%%%%%%%%%%%%
	\begin{rmk}
			\label{rmk.axioms}
		It is clear that Theorem \ref{thm.criterion}
		holds for any triangulated functor
			\[	F:\stablehomotopyY \rightarrow \stablehomotopyX
			\]
		which satisfies the following axioms:
		\begin{enumerate}
		 \item \label{rmk.axioms1}For every 
			$q\in \mathbb Z$, $F(\Sigma _{T}^{q}\stablehomotopyY^{\mathit{eff}})\subseteq
			\Sigma _{T}^{q}\stablehomotopyX^{\mathit{eff}}$.
		 \item \label{rmk.axioms2}$F$ commutes with filtered homotopy colimits.
		\end{enumerate}
		Interesting examples are the following:
		\begin{enumerate}
		 \item \label{rmk.axioms3}$A\wedge ^{\mathbf L}-:\stablehomotopyX \rightarrow \stablehomotopyX$,
			where $A$ is a symmetric $T_{X}$-spectrum in $\stablehomotopyX^{eff}$.
		 \item \label{rmk.axioms4}$\mathbf{L}g_{\sharp}:\stablehomotopyX \rightarrow \stablehomotopyY$,
			where $g:X\rightarrow Y$ is a smooth map of finite type.			 
		\end{enumerate}
	\end{rmk}
	
%%%%%%%%%%%%%%%%%%%%%%%%%%%%%%%%%%%%%%%%%%%%%%%%%%%%%%%%
%%%%%%%%%%%%%%%%%%%%%%%%%%%%%%%%%%%%%%%%%%%%%%%%%%%%%%%%
%%%%%%%%%%%%%%%%%%%%%%%%%%%%%%%%%%%%%%%%%%%%%%%%%%%%%%%%
\begin{rmk}
		\label{remark.slices-MGLmodules}
	For the applications in this paper, we will not need the full force of Theorem \ref{thm.criterion}
	since we will prove a stronger statement, i.e. that the condition (\ref{equation.orthogonalitycondition})
	holds for every symmetric $T$-spectrum $E$ in $\qplusoneorthogonalX$.
	However, Theorem \ref{thm.criterion} is still interesting, since the slices have much more structure 
	and nicer properties, for instance they are always modules in $Spt(\mathcal M _{X})$ over Voevodsky's
	algebraic cobordism spectrum $MGL$ (see \cite{oriensf}).
	We refer the reader to \cite{SKthesis}
	for some interesting applications of Theorem \ref{thm.criterion}.
\end{rmk}	
	
	\begin{prop}
			\label{prop.smoothnes=orth.comp}
		Assume that $g:X\rightarrow Y$ is a smooth map of finite type.  Let
		$q\in \mathbb Z$ be an arbitrary integer, and  $E\in \qorthogonalY$
		an arbitrary symmetric $T_{Y}$-spectrum.  Then
			\[	\mathbf L g^{\ast}E\in \qorthogonalX
			\]
	\end{prop}	
	\begin{proof}
		Since $g$ is smooth, the functor $\mathbf L g^{\ast}$ admits a
		left adjoint $\mathbf L g_{\sharp}$.  Then, the result follows
		immediately from adjointness.
	\end{proof}
%%%%%%%%%%%%%%%%%%%%%%%%%%%%%%%%%%%%%%%%%%%%%%%%%%%%%%%%%%%%%%%%%%
%%%%%%%%%%%%%%%%%%%%%%%%%%%%%%%%%%%%%%%%%%%%%%%%%%%%%%%%%%%%%%%%%%
%%%%%%%%%%%%%%%%%%%%%%%%%%%%%%%%%%%%%%%%%%%%%%%%%%%%%%%%%%%%%%%%%%
	\begin{cor}
			\label{cor.smoothnes=>compatibility}
		Assume that $g:X\rightarrow Y$ is a smooth map of finite type.  Then
		for every symmetric $T_{Y}$-spectrum in
		$\stablehomotopyY$ and for every integer $\ell \in \mathbb Z$,
		the condition (\ref{equation.orthogonalitycondition}) in Lemma \ref{lemma.criterion} holds;
		and as a consequence the slice filtration is compatible
		with pullbacks along $g$ in the sense of Definition \ref{def.compatibility-with-pullbacks}.
	\end{cor}
	\begin{proof}
		Consider a symmetric $T_{Y}$-spectrum $E$ in $\stablehomotopyY$ and fix
		an integer $q\in \mathbb Z$.  By construction, $s_{q}E\in 
		\qplusoneorthogonalY$.  Thus the result follows directly from
		Proposition \ref{prop.smoothnes=orth.comp} and Theorem \ref{thm.criterion}.
	\end{proof}
	
\end{section}

%% file: sect2_fsf.tex
%%%%%%%%%%%%%%%%%%%%%%%%%%%%%%%%%%%%%%%%%%%%%%%%%%%%%%%
%%%%%%%%%%%%%%%%%%%%%%%%%%%%%%%%%%%%%%%%%%%%%%%%%%%%%%%
%%%%%%%%%%%%%%%%%%%%%%%%%%%%%%%%%%%%%%%%%%%%%%%%%%%%%%%
\begin{section}{The case of schemes defined over a field with resolution of singularities}
		\label{sect-2}
		
	In this section $k$ will denote a  field with resolution of singularities and
	$X$ will be a separated $k$-scheme of finite type with structure map
	$g:X\rightarrow \mathrm{Spec}\; k$.  
	Our goal is to
	show that the condition (\ref{equation.orthogonalitycondition}) of
	Lemma \ref{lemma.criterion} holds for every symmetric $T_{k}$-spectrum in $\stablehomotopyk$
	and
	for every integer $q\in \mathbb Z$.  Thus, 
	by Theorem \ref{thm.criterion} we conclude
	that in this situation there exists compatibility between
	the slice filtration and pullback along $g$
	in the sense of Definition \ref{def.compatibility-with-pullbacks}.
	
\begin{defi}
		\label{def.res.sings}
	We will say that a  field $k$ admits resolution of singularities if the following
	condition holds:
		\begin{description}
			\item[RS] For any separated $k$-scheme of finite type $X$, there exists a proper
								and birational morphism $p:\tilde{X}\rightarrow X$ such that
								$\tilde{X}$ is smooth over $k$.
		\end{description}
\end{defi}

\begin{rmk}
		\label{rmk.perfect.field}
	Notice that if a field $k$ admits resolution of singularities, then in particular it
	is a perfect field.
\end{rmk}
	
%%%%%%%%%%%%%%%%%%%%%%%%%%%%%%%%%%%%%%%%%%%%%%%%%%%%%%%
%%%%%%%%%%%%%%%%%%%%%%%%%%%%%%%%%%%%%%%%%%%%%%%%%%%%%%%
%%%%%%%%%%%%%%%%%%%%%%%%%%%%%%%%%%%%%%%%%%%%%%%%%%%%%%%
	\begin{prop}
			\label{prop.compatibility-under-res-sing}
		Let $E$ be an arbitrary symmetric $T_{k}$-spectrum in $\stablehomotopyk$
		and $q\in \mathbb Z$ an arbitrary integer.  Then
			$$\mathbf L g^{\ast}(s_{q}E) \in \qplusoneorthogonalX
			$$
	\end{prop}	
	\begin{proof}
		By Theorem \ref{thm.localization} we can assume that $X$ is a reduced scheme.
		If $X$ is smooth over $k$, then the result follows from 
		Corollary \ref{cor.smoothnes=>compatibility}.
		In the general case, we will proceed by induction on the dimension of $X$. 
		
		If $\mathrm{dim}\;X=0$, then $X$ is smooth since $k$ is in particular
		a perfect field (and $X$ is reduced), hence the result holds.  If $\mathrm{dim}\;X>0$, then
		there exist the following fibre product diagrams,
		since our base field has resolution of singularities:
		   $$\xymatrix{p^{-1}Y \ar[r]^-{\tilde{\imath}} \ar[d]_-{\tilde{p}}& W \ar[d]^-{p} &
				p^{-1}U \ar[r]^-{\tilde{\jmath}} \ar[d]_-{h}^-{\cong} & W \ar[d]^-{p}\\
				Y \ar[r]_-{i} & X
				& U=X\backslash Y \ar[r]_-{j}& X}
		   $$
		where $Y$ is a nowhere dense closed subscheme of $X$, $p$ is proper,
		dominant and birational, W is smooth over $k$ (with structure map $g\circ p$) 
		and $h$ is an isomorphism.
		
		To simplify the notation, let $F$ be $\mathbf L (g\circ p)^{\ast}(s_{q}E)$.
		By Theorem \ref{thm.localization}, the following diagram is
		a distinguished triangle in
		$\stablehomotopyW$:
			\[	\mathbf L \tilde{\jmath}_{\sharp} \; \mathbf L \tilde{\jmath}^{\ast}(F) 
				    \rightarrow F \rightarrow 
				    \mathbf R \tilde{\imath}_{\ast} \; \mathbf L \tilde{\imath}^{\ast}(F) \rightarrow
				    S^{1}\wedge  \mathbf L \tilde{\jmath}_{\sharp}\; \mathbf L \tilde{\jmath}^{\ast}(F)
			\]
		Now, Corollary \ref{cor.smoothnes=>compatibility} 
		implies that $F=\mathbf L (g\circ p)^{\ast}(s_{q}E)$
		is in $\qplusoneorthogonalW$, since $g\circ p:W\rightarrow k$
		is a smooth map of finite type.  By induction on the dimension
		($\mathrm{dim}\; p^{-1}Y<\mathrm{dim}\; X$), we deduce that
		$\mathbf L \tilde{\imath}^{\ast}(F) \cong 
		\mathbf L (g \circ p \circ \tilde{\imath})^{\ast}(s_{q}E)$ is in
		$\qplusoneorthogonalinvY$, thus Lemma \ref{lemma.pushforward-respects-orthogonal}
		implies that $\mathbf R \tilde{\imath}_{\ast} \;
		\mathbf L \tilde{\imath}^{\ast}(F)$ is in $\qplusoneorthogonalW$.
		Therefore, it follows from Lemma \ref{lemma.orthognal-triangulated} that
		$\mathbf L \tilde{\jmath}_{\sharp} \; \mathbf L \tilde{\jmath}^{\ast}(F)$ 
		is also in $\qplusoneorthogonalW$.

		By Lemma \ref{lemma.pushforward-respects-orthogonal} we
		conclude that 
			\[ \mathbf R p_{\ast}\; 
				\mathbf L \tilde{\jmath}_{\sharp} \; \mathbf L \tilde{\jmath}^{\ast}(F) 
				\cong  \mathbf R p_{\ast}\; 
				\mathbf L \tilde{\jmath}_{\sharp} \; \mathbf L \tilde{\jmath}^{\ast}
				\; \mathbf L p^{\ast}(\mathbf L g^{\ast}s_{q}E)
			\]
		is in $\qplusoneorthogonalX$.  On the other hand, we claim 
		the existence of the following natural isomorphisms 
		in $\stablehomotopyX$:
			\begin{eqnarray}
			 	\mathbf R p_{\ast}\; 
				\mathbf L \tilde{\jmath}_{\sharp} \; (\mathbf L \tilde{\jmath}^{\ast}
				\; \mathbf L p^{\ast})\;\mathbf L g^{\ast}s_{q}E& \cong &
			\mathbf R p_{\ast}\; 
				\mathbf L \tilde{\jmath}_{\sharp} \; (\mathbf L h^{\ast}
				\; \mathbf L j^{\ast})\;\mathbf L g^{\ast}s_{q}E
			\label{iso1}\\
			&\cong &  (\mathbf L j_{\sharp}\; \mathbf R h_{\ast}) 
			\; \mathbf L h^{\ast}
				\; \mathbf L j^{\ast}\;\mathbf L g^{\ast}s_{q}E 
			\label{iso2}\\
			&\cong & \mathbf L j_{\sharp}
				\; \mathbf L j^{\ast}(\mathbf L g^{\ast}s_{q}E) 
			\label{iso3}
			\end{eqnarray}
		In effect;  (\ref{iso1}) follows from functoriality,
		(\ref{iso2}) follows from Proposition \ref{prop.sharp.pfwd.comp} and
		(\ref{iso3}) follows from the fact that $h$ is an isomorphism.
		Therefore, we conclude that $\mathbf L j_{\sharp}\; 
		\mathbf L j^{\ast}\; \mathbf L g^{\ast}s_{q}E$ 
		is in $\qplusoneorthogonalX$.  
		
		On the other hand, by induction
		on the dimension ($\mathrm{dim}\; Y<\mathrm{dim}\; X$),
		we can assume that $\mathbf L i^{\ast} (\mathbf L g^{\ast}s_{q}E)$ 
		is in $\qplusoneorthogonalY$, and using Lemma \ref{lemma.pushforward-respects-orthogonal}
		we deduce that 
		$\mathbf R i_{\ast} \; \mathbf L i^{\ast} (\mathbf L g^{\ast}s_{q}E)$
		is in $\qplusoneorthogonalX$.
	
		Finally, by Theorem \ref{thm.localization} 
		the following diagram is
		a distinguished triangle in
		$\stablehomotopyX$:
			\[	\mathbf L j_{\sharp}\; \mathbf L j^{\ast}(\mathbf L g^{\ast}s_{q}E) 
				    \rightarrow \mathbf L g^{\ast}s_{q}E \rightarrow 
				    \mathbf R i_{\ast} \; \mathbf L i^{\ast}(\mathbf L g^{\ast}s_{q}E) 
				    \rightarrow
				    S^{1}\wedge  \mathbf L j_{\sharp}\; \mathbf L j^{\ast}(\mathbf L g^{\ast}s_{q}E)
			\]
		Hence,
		Lemma \ref{lemma.orthognal-triangulated}  implies that
		$\mathbf L g^{\ast}(s_{q}E)$ is in $\qplusoneorthogonalX$,
		as we wanted.
	\end{proof}
		
%%%%%%%%%%%%%%%%%%%%%%%%%%%%%%%%%%%%%%%%%%%%%%%%%%%%%%%%%%%%%%%%%%
%%%%%%%%%%%%%%%%%%%%%%%%%%%%%%%%%%%%%%%%%%%%%%%%%%%%%%%%%%%%%%%%%%
%%%%%%%%%%%%%%%%%%%%%%%%%%%%%%%%%%%%%%%%%%%%%%%%%%%%%%%%%%%%%%%%%%
	\begin{thm}
	 	   \label{thm.compatibility-with-res-sing}
		Let
		$X$ be a separated $k$-scheme of finite type with structure map
		$g:X\rightarrow k$, where $k$ has resolution of singularities.
		Then the slice filtration is compatible with pullbacks along $g$
		in the sense of Definition \ref{def.compatibility-with-pullbacks}.
	\end{thm}
	\begin{proof}
		It follows directly from Theorem \ref{thm.criterion} together with 
		Proposition \ref{prop.compatibility-under-res-sing}.
	\end{proof}

	\begin{cor}
			\label{cor.pullback.always.comp}
		Let $E\in \stablehomotopyk$ be an arbitrary symmetric $T_{k}$-spectrum
		and $q\in \mathbb Z$ an arbitrary integer. Let
		$h:X\rightarrow Y$ be a map of separated
		$k$-schemes of finite type, with structure
		maps $u$, $v$ respectively.
		Then, there exists a
		canonical isomorphism in $\stablehomotopyX$:
			\[	\beta _{q}(\mathbf L v^{\ast}E):
				\mathbf L h^{\ast}(s_{q}\mathbf L v^{\ast} E)\rightarrow
					s_{q}(\mathbf L h^{\ast}\; \mathbf L v^{\ast}E)
					\cong s_{q}(\mathbf L u^{\ast} E)
			\]
	\end{cor}
	\begin{proof}
		By Theorem \ref{thm.compatibility-with-res-sing},
			\begin{align*}	
				\beta _{q}^{Y}(E): & \mathbf L v^{\ast} s_{q}E\rightarrow s_{q}
				\mathbf L v^{\ast}E \\
				\beta _{q}^{X}(E): & \mathbf L u^{\ast} s_{q}E
				\cong \mathbf L h^{\ast}\; \mathbf L v^{\ast}s_{q}E\rightarrow 
				s_{q} \mathbf L u^{\ast} E
			\end{align*}
		are isomorphisms in $\stablehomotopyY$ and $\stablehomotopyX$
		respectively.  Thus, we deduce that
		$\mathbf L h^{\ast} (\beta _{q}^{Y}(E))$
		is an isomorphism in $\stablehomotopyX$.  Finally, we observe
		that the following diagram in $\stablehomotopyX$ commutes
			\[	\xymatrix{\mathbf L h^{\ast}(s_{q}\mathbf L v^{\ast} E) 
					\ar[rr]^-{\beta _{q}(\mathbf L v^{\ast}E)}&&
									s_{q}(\mathbf L h^{\ast}\; \mathbf L v^{\ast}E)
					\cong s_{q}(\mathbf L u^{\ast} E)\\
					 \mathbf L u^{\ast} s_{q}E
				\cong \mathbf L h^{\ast}\; \mathbf L v^{\ast}s_{q}E
				\ar[u]^-{\mathbf L h^{\ast} (\beta _{q}^{Y}(E))} \ar[urr]_-{\beta _{q}^{X}(E)}
				&&
					}
			\]
		Hence the result follows.
	\end{proof}
	
\end{section}

%% file: sect3_fsf.tex
%%%%%%%%%%%%%%%%%%%%%%%%%%%%%%%%%%%%%%%%%%%%%%%%%%%%%%%%%%%
%%%%%%%%%%%%%%%%%%%%%%%%%%%%%%%%%%%%%%%%%%%%%%%%%%%%%%%%%%%
%%%%%%%%%%%%%%%%%%%%%%%%%%%%%%%%%%%%%%%%%%%%%%%%%%%%%%%%%%%
\begin{section}{Applications}
		\label{sect-3}

	In this section we assume that all our schemes are of finite type
	over a field $k$ of characteristic zero.

%%%%%%%%%%%%%%%%%%%%%%%%%%%%%%%%%%%%%%%%%%%%%%%%%%%%%%%%%%%%%%%%%%
%%%%%%%%%%%%%%%%%%%%%%%%%%%%%%%%%%%%%%%%%%%%%%%%%%%%%%%%%%%%%%%%%%
%%%%%%%%%%%%%%%%%%%%%%%%%%%%%%%%%%%%%%%%%%%%%%%%%%%%%%%%%%%%%%%%%%
	\begin{defi}
	 	   \label{def.notation}
	 	Let 
		$\mathbf 1 _{X}$, $\mathbf{KH} _{X}$, $\mathbf{HZ} _{X}$,
		$\mathbf{HZ} _{X}^{\slicefilt}  \in Spt(\mathcal M _{X})$ denote
		respectively
		the sphere spectrum,  the spectrum
		representing Weibel's homotopy invariant $K$-theory
		\cite{MR991991},
		the spectrum representing motivic cohomology \cite{MR2029171}
		and $s_{0}(\mathbf 1 _{X})$.
	\end{defi}
	
	Theorems \ref{thm.computation-slices} and \ref{thm.slices-modules}
	prove several conjectures of
	Voevodsky \cite[conjectures 1, 7, 10, 11]{MR1977582} in characteristic zero.

%%%%%%%%%%%%%%%%%%%%%%%%%%%%%%%%%%%%%%%%%%%%%%%%%%%%%%%%%%%%%%%%%%
%%%%%%%%%%%%%%%%%%%%%%%%%%%%%%%%%%%%%%%%%%%%%%%%%%%%%%%%%%%%%%%%%%
%%%%%%%%%%%%%%%%%%%%%%%%%%%%%%%%%%%%%%%%%%%%%%%%%%%%%%%%%%%%%%%%%%
	\begin{thm}
	 	   \label{thm.computation-slices}
		Let
		$X$ be a separated $k$-scheme of finite type with structure map
		$g:X\rightarrow k$.  Then:
		   \begin{enumerate}
		    \item \label{thm.computation-slices.1}The zero slice of the sphere spectrum,
							$\mathbf{HZ} _{X}^{\slicefilt}$ is isomorphic to
							$\mathbf L g^{\ast} (\mathbf{HZ}_{k})$ in $\stablehomotopyX$.
		    \item   \label{thm.computation-slices.2}The zero slice of the sphere spectrum,
							$\mathbf{HZ} _{X}^{\slicefilt}$ is a 
							cofibrant ring spectrum in $Spt(\mathcal M _{X})$.
			\item   \label{thm.computation-slices.3}The zero slice of the sphere spectrum,
							$\mathbf{HZ} _{X}^{\slicefilt}$ is an $E_{\infty}$-ring 
							spectrum in $\stablehomotopyX$.  Moreover,
							if $X$ is smooth then 
							$\mathbf{HZ} _{X}^{\slicefilt}$ is a 
							commutative ring spectrum in $\TspectraX$.
		    \item   \label{thm.computation-slices.4}For every integer $q$, 
							$s_{q}(\mathbf{KH} _{X})$ is isomorphic to 
							$T^{q}_{X}\wedge \mathbf{HZ} _{X}^{\slicefilt}$
							in $\stablehomotopyX$.
		    \item   \label{thm.computation-slices.5}If we consider rational coefficients
							and $X$ is geometrically unibranch then
							$\mathbf{HZ}_{X}^{\slicefilt} \otimes \mathbb Q$,
							$s_{q}(\mathbf{KH} _{X})\otimes \mathbb Q$
							are respectively isomorphic in $\stablehomotopyX$ to 
							$\mathbf{HZ}_{X}\otimes \mathbb Q$,
							$(T^{q}_{X}\wedge \mathbf{HZ}_{X})\otimes \mathbb Q$.
		   \end{enumerate}
	\end{thm}
	\begin{proof}
		(\ref{thm.computation-slices.1}):  It is clear that
		$\mathbf 1 _{X}\cong \mathbf L g^{\ast}(\mathbf 1 _{k})$ in
		$\stablehomotopyX$.   Therefore, by Theorem \ref{thm.compatibility-with-res-sing} we deduce
		the existence of
		the following natural isomorphisms in $\stablehomotopyX$
			\[	s_{0}(\mathbf 1 _{X})\cong s_{0}(\mathbf L g^{\ast} \mathbf 1 _{k})
				\cong \mathbf L g^{\ast}(s_{0}\mathbf 1 _{k})
			\]
		Finally, the result follows from the work of Levine \cite[Thm. 10.5.1]{MR2365658}
		and Voevodsky 
		\cite[Thm. 6.6]{MR2101286}, which implies that the unit map $u:\mathbf 1 _{k}\rightarrow
		\mathbf{HZ}_{k}$ induces the following
		isomorphisms in $\stablehomotopyk$
			\[	s_{0}(u):s_{0}\mathbf 1 _{k}
				 \rightarrow s_{0}\mathbf{HZ} _{k}\cong \mathbf{HZ} _{k}
			\]
		(\ref{thm.computation-slices.2}):  We observe that
		$\mathbf{HZ}_{k}$ is a ring spectrum in
		$Spt(\mathcal M _{X})$ (see \cite[Lemma 4.6]{MR2029171}).  Moreover,
		by \cite[Thm. 4.1(3)]{MR1734325},
		\cite[Thm. A.38]{MR2597741}
		and \cite[Prop. 4.19]{MR1787949},
		there exists
		a weak equivalence  
			$$w:\mathbf{HZ}_{k}^{c}\rightarrow \mathbf{HZ}_{k}$$
		in $Spt(\mathcal M _{k})$ such that 
		$\mathbf{HZ}_{k}^{c}$ is a cofibrant
		ring spectrum in $Spt(\mathcal M _{k})$.  
		On the other hand, proposition A.47 in \cite{MR2597741} implies 
		that
			\[	g^{\ast}:Spt(\mathcal M _{k})\rightarrow Spt(\mathcal M _{X})
			\]
		is a strict symmetric monoidal left Quillen functor.  Therefore,
		$g^{\ast} (\mathbf{HZ}_{k}^{c})$ is a cofibrant ring spectrum in $Spt(\mathcal M _{X})$
		which is isomorphic to
		$\mathbf L g^{\ast} (\mathbf{HZ}_{k})$
		in $\stablehomotopyX$.  Thus, the result follows from
		(\ref{thm.computation-slices.1}) above.
		
		(\ref{thm.computation-slices.3}):  The fact that $\mathbf{HZ} _{X}^{\slicefilt}$ is an 
		$E_{\infty}$-ring spectrum in $\stablehomotopyX$ follows from \cite{Gutierrez:2010fk}.
		On the other hand, if the map $g$ is smooth, then 
		$\mathbf L g^{\ast}=g^{*}$ since $\mathbf L g^{\ast}$ admits
		a left adjoint $\mathbf L g _{\sharp}$ (see 
		\cite[p. 104: Cor. 1.24]{MR1813224} and
		\cite[p. 108: line 3 and Prop. 2.9]{MR1813224}).
		By \cite[Lemma 4.6]{MR2029171}, $\mathbf{HZ}_{k}$ is a commutative ring spectrum in
		$Spt(\mathcal M _{X})$.
		Thus, $g^{\ast} (\mathbf{HZ}_{k})$ is a commutative ring spectrum in $Spt(\mathcal M _{X})$
		which is isomorphic to
		$\mathbf L g^{\ast} (\mathbf{HZ}_{k})$
		in $\stablehomotopyX$.  Finally, the result follows from
		(\ref{thm.computation-slices.1}) above.

		(\ref{thm.computation-slices.4}):  It follows from 
		\cite[section 6.2]{MR1648048} 
		(see also \cite[Thm. 2.15 and Prop. 3.8]{Cisinski:2010fk}) that
		$\mathbf{KH}_{X}=\mathbf L g^{\ast}(\mathbf{KH}_{k})$.  Now, by
		Theorem \ref{thm.compatibility-with-res-sing} there exist the following
		natural isomorphisms in $\stablehomotopyX$
			\[	s_{q}\mathbf{KH}_{X}\cong s_{q}(\mathbf L g^{\ast}\mathbf{KH}_{k})
				\cong \mathbf L g^{\ast}(s_{q}\mathbf{KH}_{k})
			\]
		Finally, the work of Levine 
		\cite[Thms. 6.4.2 and 9.0.3]{MR2365658} implies that
		$s_{q}\mathbf{KH}_{k}$ is isomorphic in $\stablehomotopyk$ to
		$T^{q}_{k}\wedge\mathbf{HZ}_{k}$.  Thus
			\begin{align*}	s_{q}\mathbf{KH}_{X} &
								\cong \mathbf L g^{\ast}(s_{q}\mathbf{KH}_{k})
								\cong \mathbf L g^{\ast}(T^{q}_{k}\wedge \mathbf{HZ}_{k}) \\
								& \cong T^{q}_{X}\wedge \mathbf L g^{\ast}(\mathbf{HZ}_{k})
								\cong T^{q}_{X}\wedge \mathbf{HZ}_{X}^{\slicefilt}
			\end{align*}
		as we wanted.

		(\ref{thm.computation-slices.5}):  The work of Cisinski and D{\'e}glise 
		\cite[Cor. 15.1.6(2)]{mixedmotives} implies that under these conditons 
		$\mathbf L g^{\ast}(\mathbf{HZ}_{k})\otimes \mathbb Q$
		is isomorphic to $\mathbf{HZ}_{X}\otimes \mathbb Q$ in $\stablehomotopyX$.
		Therefore, the result follows from (\ref{thm.computation-slices.1}) 
		and (\ref{thm.computation-slices.4}) above.
	\end{proof}
	
	\begin{cor}
			\label{cor.mainthm}
		Let  $h:X\rightarrow Y$ be a map of separated
		$k$-schemes of finite type, and
		$q\in \mathbb Z$ an arbitrary integer.
		Then:
			\begin{enumerate}
				\item \label{cor.mainthm.a}  There exists a canonical isomorphism
					in $\stablehomotopyX$
					\[	\mathbf L h^{\ast}(\mathbf{HZ} _{Y}^{\slicefilt})\cong
						\mathbf{HZ} _{X}^{\slicefilt}
					\]
				\item \label{cor.mainthm.b}  There exists a canonical isomorphism
					in $\stablehomotopyX$
					\[	\mathbf L h^{\ast}(s_{q}(\mathbf{KH} _{Y}))\cong
						s_{q}(\mathbf{HZ} _{X}^{\slicefilt})
					\]
			\end{enumerate}
	\end{cor}
	\begin{proof}
		This follows directly from Corollary \ref{cor.pullback.always.comp} together with
		\eqref{thm.computation-slices.1} and \eqref{thm.computation-slices.4}
		in Theorem \ref{thm.computation-slices}.
	\end{proof}

%%%%%%%%%%%%%%%%%%%%%%%%%%%%%%%%%%%%%%%%%%%%%%%%%%%%%%%%%%%%%%%%%%
%%%%%%%%%%%%%%%%%%%%%%%%%%%%%%%%%%%%%%%%%%%%%%%%%%%%%%%%%%%%%%%%%%
%%%%%%%%%%%%%%%%%%%%%%%%%%%%%%%%%%%%%%%%%%%%%%%%%%%%%%%%%%%%%%%%%%
	\begin{rmk}
	 	   \label{rmk.computation-slices}
	We may consider Theorem
	\ref{thm.computation-slices}\eqref{thm.computation-slices.4} as an extension
	of the computation of Levine \cite[Thms. 6.4.2 and 9.0.3]{MR2365658}
	from fields to schemes of finite type, however
	notice that we need to assume that our base scheme is defined over a field of characteristic zero
	whereas \cite{MR2365658}  holds over perfect fields.  

	Similarly, we may consider
	Theorem \ref{thm.computation-slices}\eqref{thm.computation-slices.1}
	as an extension of the computation of 
	Voevodsky \cite[Thm. 6.6]{MR2101286} 
	and Levine \cite[Thm. 10.5.1]{MR2365658}, but
	\cite{MR2365658} also holds over perfect fields whereas we need to assume that
	our base scheme is defined over a field of characteristic zero.
	\end{rmk}

%%%%%%%%%%%%%%%%%%%%%%%%%%%%%%%%%%%%%%%%%%%%%%%%%%%%%%%%%%%%%%%%%%
%%%%%%%%%%%%%%%%%%%%%%%%%%%%%%%%%%%%%%%%%%%%%%%%%%%%%%%%%%%%%%%%%%
%%%%%%%%%%%%%%%%%%%%%%%%%%%%%%%%%%%%%%%%%%%%%%%%%%%%%%%%%%%%%%%%%%
	\begin{thm}
	 	\label{thm.slices-modules}
	  Let $E$ be an arbitrary symmetric $T_{X}$-spectrum
	  in $Spt(\mathcal M _{X})$ and $q\in \mathbb Z$
	  an arbitrary integer.
		\begin{enumerate}
		 \item \label{thm.slices-modules.a}  The $q$-slice of $E$, $s_{q}(E)$
			has a natural structure of
			$\mathbf{HZ}_{X}^{\slicefilt}$-module in $Spt(\mathcal M _{X})$.
		 \item \label{thm.slices-modules.b}  If we consider rational coefficients
			and $X$ is geometrically unibranch then $s_{q}(E)\otimes \mathbb Q$
			has a natural structure
			of $\mathbf{HZ}_{X}\otimes \mathbb Q$-module in $Spt(\mathcal M _{X})$,
			in particular $s_{q}(E)\otimes \mathbb Q$ has transfers.
		\end{enumerate}

	\end{thm}
	\begin{proof}
	   By construction,  $Spt(\mathcal M _{X})$
	   is cellular \cite{MR1944041} 
	   and the spectra
	   $F_{n}(S^{r}\wedge \mathbb G _{m}^{s}\wedge U_{+})$ are all cofibrant in $Spt(\mathcal M _{X})$
	   for every $U\in Sm_{X}$ and integers $n,r,s\geq 0$
	   (see \cite[Lem. A.10]{MR2597741}).
	   
	   Therefore, \cite[Thm. 2.1]{MR2576905} and
	   \cite[Lem. 3.6.21(3) and Thm. 3.6.20]{MR2807904}
	   hold in $Spt(\mathcal M _{X})$.  Then, the result follows directly
	   from Theorem \ref{thm.computation-slices}.
	\end{proof}

%%%%%%%%%%%%%%%%%%%%%%%%%%%%%%%%%%%%%%%%%%%%%%%%%%%%%%%%%%%%%%%%%%
%%%%%%%%%%%%%%%%%%%%%%%%%%%%%%%%%%%%%%%%%%%%%%%%%%%%%%%%%%%%%%%%%%
%%%%%%%%%%%%%%%%%%%%%%%%%%%%%%%%%%%%%%%%%%%%%%%%%%%%%%%%%%%%%%%%%%
	\begin{defi}
		\label{def.slicefilt-motives}
	   Let $\mathbf{HZ}_{X}^{\slicefilt}\text{-}\mathrm{mod}$ be the category of left 
	   $\mathbf{HZ}_{X}^{\slicefilt}$-modules in $Spt(\mathcal M _{X})$ equipped with the model structure
	   induced by the adjuntion
		$$(\mathbf{HZ}_{X}^{\slicefilt}\wedge -,U,\varphi):Spt(\mathcal M _{X})\rightarrow 
			\mathbf{HZ}_{X}^{\slicefilt}\text{-}\mathrm{mod}
		$$
	  i.e. a map $f$ in $\mathbf{HZ}_{X}^{\slicefilt}\text{-}\mathrm{mod}$ 
	  is a fibration or a weak equivalence
	  if and only if $Uf$ is a fibration or a weak equivalence in $Spt(\mathcal M _{X})$.
	  Let $DM_{X}^{\slicefilt}$ denote the homotopy category of
	  $\mathbf{HZ}_{X}^{\slicefilt}\text{-}\mathrm{mod}$, which is triangulated.
	\end{defi}

%%%%%%%%%%%%%%%%%%%%%%%%%%%%%%%%%%%%%%%%%%%%%%%%%%%%%%%%%%%%%%%%%%
%%%%%%%%%%%%%%%%%%%%%%%%%%%%%%%%%%%%%%%%%%%%%%%%%%%%%%%%%%%%%%%%%%
%%%%%%%%%%%%%%%%%%%%%%%%%%%%%%%%%%%%%%%%%%%%%%%%%%%%%%%%%%%%%%%%%%
	\begin{thm}
			\label{thm.motivic-cats}
		The $2$-functor $X\mapsto DM_{X}^{\slicefilt}$ has the structure of a motivic
		category in the sense of Cisinski and D{\'e}glise \cite{mixedmotives},
		and the adjunction
			\[	(\mathbf{HZ}_{X}^{\slicefilt}\wedge ^{\mathbf{L}} -,\mathbf{R}U,\varphi):\stablehomotopyX \rightarrow 
			DM_{X}^{\slicefilt}
			\]
		is a morphism of motivic categories $\mathcal{SH}\rightarrow DM^{\slicefilt}$
		in the category $Sch_{K}$ of separated $k$-schemes of finite type.
		
		In particular, $X\mapsto DM_{X}^{\slicefilt}$ is a closed symmetric monoidal
		homotopic stable $2$-functor
		in the sense of Ayoub, i.e.
		given a map $g$ in $Sch_{k}$ the functors $\mathbf L g^{\ast}$,
		$\mathbf R g_{\ast}$, $g_{!}$, $g^{!}$ exist and satisfy
		the formalism of \cite[Scholium 1.4.2]{MR2423375}.  Moreover, $DM_{X}^{\slicefilt}$
		is a closed symmetric triangulated category satisfying the
		formalism of  \cite[Chapter 2]{MR2423375}.
	\end{thm}
	\begin{proof}
		Theorem \ref{thm.computation-slices}(\ref{thm.computation-slices.1})-(\ref{thm.computation-slices.2})
		implies that $X\mapsto \mathbf{HZ}_{X}^{\slicefilt}$ is a family of  
		cofibrant ring spectra in $Spt(\mathcal M _{X})$
		which is stable under pullback in the category of separated $k$-schemes of finite type.  
		Hence Propositions 4.2.11, 4.2.16 
		and Corollary 2.4.9 in \cite{mixedmotives} imply that
		$(\mathbf{HZ}_{X}^{\slicefilt}\wedge ^{\mathbf{L}} -,\mathbf{R}U,\varphi)$
		is a morphism of motivic categories and that 
		 $X\mapsto DM_{X}^{\slicefilt}$ is a 
		homotopic stable $2$-functor
		in the sense of Ayoub.
		Finally, \eqref{thm.computation-slices.2} and \eqref{thm.computation-slices.3} in
		Theorem \ref{thm.computation-slices} imply that
		$DM_{X}^{\slicefilt}$
		is a closed symmetric triangulated category.
	\end{proof}
	
%%%%%%%%%%%%%%%%%%%%%%%%%%%%%%%%%%%%%%%%%%%%%%%%%%%%%%%%%%%%%%%%%%
%%%%%%%%%%%%%%%%%%%%%%%%%%%%%%%%%%%%%%%%%%%%%%%%%%%%%%%%%%%%%%%%%%
%%%%%%%%%%%%%%%%%%%%%%%%%%%%%%%%%%%%%%%%%%%%%%%%%%%%%%%%%%%%%%%%%%
	\begin{thm}
			\label{thm.comparing-motcats}
		If our base scheme is a field $k$ of characteristic zero, then $DM_{k}^{\slicefilt}$ is 
		naturally equivalent as a tensor triangulated category to
		Voevodsky's big category of motives $DM_{k}$.
		
		Therefore, the $2$-functor $X\mapsto DM_{X}^{\slicefilt}$ provides a natural framework for a theory of mixed motives
		in the category of separated $k$-schemes of finite type.
	\end{thm}
	\begin{proof}
		By construction, $DM_{k}^{\slicefilt}$ is the homotopy category of 
		$\mathbf{HZ}_{k}^{\slicefilt}$-modules in $Spt(\mathcal M _{k})$,
		where $\mathbf{HZ}_{k}^{\slicefilt}$ is the zero slice of the sphere spectrum
		$s_{0}(\mathbf  1 _{k})$.  On the other hand,  
		it follows from \cite[Thm. 10.5.1]{MR2365658}, 
		\cite[Thm. 6.6]{MR2101286} that the unit map 
		$u:\mathbf 1 _{k}\rightarrow \mathbf{HZ} _{k}$ induces a weak
		equivalence $s_{0}(u):\mathbf{HZ}_{k}^{\slicefilt}\rightarrow \mathbf{HZ}_{k}$.
		
		Thus, by 
		\cite[Prop. 2.8.5]{MR2807904}
		and \cite[Thm. A.38]{MR2597741} we deduce that 
		$DM_{k}^{\slicefilt}$ is naturally equivalent as
		a tensor triangulated category to the homotopy category of
		$\mathbf{HZ}_{k}$-modules in $Spt(\mathcal M _{k})$.  Finally,
		it follows from \cite[Thm. 1]{MR2435654} that
		Voevodsky's category of motives $DM_{k}$ is naturally equivalent
		as a tensor triangulated category to the homotopy category
		of $\mathbf{HZ}_{k}$-modules in $Spt(\mathcal M _{k})$.
	\end{proof}
	
	Let $\mathbf{H}_{\mathbb B ,X} \in Spt(\mathcal M _{X})$ denote the Beilinson motivic
	cohomology spectrum introduced by Cisinski and D{\'e}glise \cite[Def. 13.1.2]{mixedmotives}.
	It follows in particular from
	Corollary 13.2.6 in \cite{mixedmotives}  that $\mathbf{H}_{\mathbb B ,X}$ is a commutative
	cofibrant ring spectrum in $Spt(\mathcal M _{X})$ which is stable under pullback in the category
	of separated schemes of finite type over $k$.
	
%%%%%%%%%%%%%%%%%%%%%%%%%%%%%%%%%%%%%%%%%%%%%%%%%%%%%%%%%%%%%%%%%%
%%%%%%%%%%%%%%%%%%%%%%%%%%%%%%%%%%%%%%%%%%%%%%%%%%%%%%%%%%%%%%%%%%
%%%%%%%%%%%%%%%%%%%%%%%%%%%%%%%%%%%%%%%%%%%%%%%%%%%%%%%%%%%%%%%%%%
	\begin{thm}
			\label{thm.Beilinson-motives}
		The Beilinson motivic cohomology spectrum $\mathbf{H}_{\mathbb B ,X}$ is
		naturally isomorphic to $\mathbf{HZ} ^{\slicefilt}_{X}\otimes \mathbb Q$ in $\stablehomotopyX$, thus
		the homotopy category of $\mathbf{H}_{\mathbb B ,X}$-modules $\mathrm{Ho}(\mathbf{H}_{\mathbb B ,X})$
		is equivalent to the homotopy category of left $\mathbf{HZ} ^{\slicefilt}_{X}$-modules with rational
		coefficients.  
		
		Hence, we conclude that modulo torsion 
		$\mathrm{Ho}(\mathbf{H}_{\mathbb B ,X})$ and $DM_{X}^{\slicefilt}$
		are equivalent as tensor triangulated categories.
	\end{thm}
	\begin{proof}
		By
		\cite[Prop. 2.8.5]{MR2807904}
		and \cite[Thm. A.38]{MR2597741}, it suffices to prove that
		$\mathbf{H}_{\mathbb B ,X}$ is
		naturally isomorphic to $\mathbf{HZ} ^{\slicefilt}_{X}\otimes \mathbb Q$ in $\stablehomotopyX$.
		
		It follows from Theorem \ref{thm.computation-slices}(\ref{thm.computation-slices.1}) 
		that $\mathbf{HZ} ^{\slicefilt}_{X}\otimes \mathbb Q$ is stable under pullback in the category
		of separated schemes of finite type over $k$, 
		on the other hand Corollary 13.2.6 
		in \cite{mixedmotives} implies in particular
		that $\mathbf{H}_{\mathbb B ,X}$ is also stable under pullback.
		Therefore, it suffices to show that $\mathbf{H}_{\mathbb B ,k}$ and 
		$\mathbf{HZ} ^{\slicefilt}_{k}\otimes \mathbb Q$ are isomorphic in $\stablehomotopyk$
		for the base field $k$.
		
		However, Corollary 15.1.6(1) in \cite{mixedmotives} implies that $\mathbf{H}_{\mathbb B ,k}$ and
		$\mathbf{HZ}_{k}\otimes \mathbb Q$ are naturally isomorphic in $\stablehomotopyk$,
		and finally it follows from Theorem \ref{thm.computation-slices}(\ref{thm.computation-slices.1}) 
		that $\mathbf{HZ}_{k}\otimes \mathbb Q$ and $\mathbf{HZ} ^{\slicefilt}_{k}\otimes \mathbb Q$ 
		are also naturally isomorphic in $\stablehomotopyk$.
		This finishes the proof.
	\end{proof}

\end{section}

%% file: fsf.bbl
\begin{thebibliography}{10}

\bibitem{MR2423375}
J.~Ayoub.
\newblock Les six op\'erations de {G}rothendieck et le formalisme des cycles
  \'evanescents dans le monde motivique. {I}.
\newblock {\em Ast\'erisque}, (314):x+466 pp. (2008), 2007.

\bibitem{MR2438151}
J.~Ayoub.
\newblock Les six op{\'e}rations de {G}rothendieck et le formalisme des cycles
  {\'e}vanescents dans le monde motivique. {II}.
\newblock {\em Ast{\'e}risque}, (315):vi+364 pp. (2008), 2007.

\bibitem{Cisinski:2010fk}
D.-C. Cisinski.
\newblock Descente propre en k-th{\'e}orie invariante par homotopie.
\newblock {\em preprint}, 2010.

\bibitem{mixedmotives}
D.-C. Cisinski and F.~D{\'e}glise.
\newblock Triangulated categories of mixed motives.
\newblock {\em preprint}, 2009.

\bibitem{MR2029171}
B.~I. Dundas, O.~R{\"o}ndigs, and P.~A. {\O}stv{\ae}r.
\newblock Motivic functors.
\newblock {\em Doc. Math.}, 8:489--525 (electronic), 2003.

\bibitem{Gutierrez:2010fk}
J.~J. Guti{\'e}rrez, O.~R{\"o}ndigs, M.~Spitzweck, and P.~A. {\O}stv{\ae}r.
\newblock Motivic slices and colored operads.
\newblock {\em preprint}, 2010.

\bibitem{MR1944041}
P.~S. Hirschhorn.
\newblock {\em Model categories and their localizations}, volume~99 of {\em
  Mathematical Surveys and Monographs}.
\newblock American Mathematical Society, Providence, RI, 2003.

\bibitem{MR1787949}
J.~F. Jardine.
\newblock Motivic symmetric spectra.
\newblock {\em Doc. Math.}, 5:445--553 (electronic), 2000.

\bibitem{SKthesis}
S.~Kelly.
\newblock {\em Triangulated categories of motives in positive characteristic.}
\newblock PhD thesis, Paris 13, 2012.

\bibitem{MR2365658}
M.~Levine.
\newblock The homotopy coniveau tower.
\newblock {\em J. Topol.}, 1(1):217--267, 2008.

\bibitem{MR1813224}
F.~Morel and V.~Voevodsky.
\newblock {${\bf A}\sp 1$}-homotopy theory of schemes.
\newblock {\em Inst. Hautes \'Etudes Sci. Publ. Math.}, (90):45--143 (2001),
  1999.

\bibitem{MR1308405}
A.~Neeman.
\newblock The {G}rothendieck duality theorem via {B}ousfield's techniques and
  {B}rown representability.
\newblock {\em J. Amer. Math. Soc.}, 9(1):205--236, 1996.

\bibitem{MR1812507}
A.~Neeman.
\newblock {\em Triangulated categories}, volume 148 of {\em Annals of
  Mathematics Studies}.
\newblock Princeton University Press, Princeton, NJ, 2001.

\bibitem{MR2597741}
I.~Panin, K.~Pimenov, and O.~R{{\"o}}ndigs.
\newblock On {V}oevodsky's algebraic {$K$}-theory spectrum.
\newblock In {\em Algebraic topology}, volume~4 of {\em Abel Symp.}, pages
  279--330. Springer, Berlin, 2009.

\bibitem{MR2576905}
P.~Pelaez.
\newblock Mixed motives and the slice filtration.
\newblock {\em C. R. Math. Acad. Sci. Paris}, 347(9-10):541--544, 2009.

\bibitem{MR2807904}
P.~Pelaez.
\newblock Multiplicative properties of the slice filtration.
\newblock {\em Ast{\'e}risque}, (335):xvi+289, 2011.

\bibitem{oriensf}
P.~Pelaez.
\newblock On the orientability of the slice filtration.
\newblock {\em Homology, Homotopy Appl.}, 13(2):293--300, 2011.

\bibitem{MR0223432}
D.~G. Quillen.
\newblock {\em Homotopical algebra}.
\newblock Lecture Notes in Mathematics, No. 43. Springer-Verlag, Berlin, 1967.

\bibitem{MR2435654}
O.~R{{\"o}}ndigs and P.~A. {\O}stv{\ae}r.
\newblock Modules over motivic cohomology.
\newblock {\em Adv. Math.}, 219(2):689--727, 2008.

\bibitem{MR1734325}
S.~Schwede and B.~E. Shipley.
\newblock Algebras and modules in monoidal model categories.
\newblock {\em Proc. London Math. Soc. (3)}, 80(2):491--511, 2000.

\bibitem{MR1648048}
V.~Voevodsky.
\newblock {$\bold A^1$}-homotopy theory.
\newblock In {\em Proceedings of the {I}nternational {C}ongress of
  {M}athematicians, {V}ol. {I} ({B}erlin, 1998)}, number Extra Vol. I, pages
  579--604 (electronic), 1998.

\bibitem{MR1977582}
V.~Voevodsky.
\newblock Open problems in the motivic stable homotopy theory. {I}.
\newblock In {\em Motives, polylogarithms and Hodge theory, Part I (Irvine, CA,
  1998)}, volume~3 of {\em Int. Press Lect. Ser.}, pages 3--34. Int. Press,
  Somerville, MA, 2002.

\bibitem{MR2101286}
V.~Voevodsky.
\newblock On the zero slice of the sphere spectrum.
\newblock {\em Tr. Mat. Inst. Steklova}, 246(Algebr. Geom. Metody, Svyazi i
  Prilozh.):106--115, 2004.

\bibitem{MR991991}
C.~A. Weibel.
\newblock Homotopy algebraic {$K$}-theory.
\newblock In {\em Algebraic {$K$}-theory and algebraic number theory
  ({H}onolulu, {HI}, 1987)}, volume~83 of {\em Contemp. Math.}, pages 461--488.
  Amer. Math. Soc., Providence, RI, 1989.

\end{thebibliography}
